\documentclass[11pt]{article}
\usepackage{geometry}                
\usepackage{graphicx}
\usepackage{amssymb}
\usepackage{epstopdf}
\usepackage{amsthm}
\usepackage{amsmath}
\usepackage{mathtools}
\usepackage{txfonts}
\usepackage{url}
\newlength\tindent
\newtheorem{theorem}{Theorem}[section]
\newtheorem{lemma}[theorem]{Lemma}
\newtheorem{proposition}[theorem]{Proposition}

\title{Diophantine Inequalities with Primes, Auxiliary Inequalities, Evaluations of the Difference between Consecutive Primes}
\author{Felix Sidokhine}

\begin{document}
\maketitle

\begin{abstract}
The goal of the present paper is to present a method of proving of Diophantine inequalities with primes through the use of auxiliary inequalities and available evaluations of the difference between consecutive primes. We study the Legendre - Ingham's problem on primes in intervals $((n - 1)^k, n^k)$ and also a problem on primes in intervals $(\frac{k-1}{k}n, \frac{k}{k-1}n)$ when $k$ is a real number. A number of the new results including an alternative proof of Ingham's theorem with the effectively computable constant and also Ingham's theorem with two primes are proved.
\end{abstract}

\section{Introduction}
Our approach to Diophantine inequalities with primes is based on using auxiliary inequalities. The auxiliary inequalities for an initial problem should be proved over primes. In the case of the problems studied in this work the auxiliary inequalities and the evaluations of the difference between consecutive primes are efficient tools.

\section{Theorems}

\begin{theorem} An interval $((n - 1)^k, n^k)$, where $k \geq \frac{40}{19}$ is any fixed real number, contains a prime $p$ for every integer $n > C$, where $C$ is an effectively computable constant.\end{theorem}
\begin{theorem} An interval $(\frac{k-1}{k}n,\frac{k}{k-1}n)$, where $k \geq 2$ is any fixed real number, contains two primes for every integer $n > C$, where $C$ is an effectively computable constant.\end{theorem}
\begin{theorem} An interval $((n - 1)^3, n^3)$ contains primes for every integer $n > C$, where $C$ is an effectively computable constant.\end{theorem}
\begin{theorem} An interval $((n - 1)^3, n^3)$ contains no less than two primes for every integer $n > C$, where C is an effectively computable constant.\end{theorem}

\section{Proof of Theorems}

\subsection{``Exponential'' Theorem}

\begin{theorem}\label{thm1} An interval $((n - 1)^k, n^k)$, where $k \geq 40/19$ is an arbitrary fixed real number, contains a prime $p$ for every integer $n > C$, where $C$ is an effectively computable constant. \end{theorem}

\begin{proposition}\label{prop1} An interval $(p - (k - 0.5)p^\frac{k-1}{k}, p)$, where $k \geq 40/19$ is any fixed real number, contains a prime $q$ for every prime $p > C$, where $C$ is an effectively computable constant. \end{proposition}

\begin{lemma}\label{lem1} Proposition \ref{prop1} is true if and only if for every pair primes $p_{n-1}, p_n$ where $p_{n-1} > C$ satisfies the following inequality:

\begin{equation}\label{eq1}
p_n - p_{n-1} < (k - 0.5) p_n^{\frac{k-1}{k}}
\end{equation}
\end{lemma}

\begin{proof}
Let proposition \ref{prop1} be true for every $p_n > C$, hence there is $q \in (p_n - (k - 0.5) p_n^{\frac{k-1}{k}}, p_n)$ such that $q \leq p_{n-1} < p_n$ and $p_{n-1}$ also belongs to this interval. Thus

\begin{equation}
p_n - p_{n-1} < p_n - (p_n - (k - 0.5)p_n^{\frac{k-1}{k}}) = (k - 0.5) p_n^\frac{k-1}{k}
\end{equation}

Let $p_n - p_{n-1} < (k - 0.5)p_n^\frac{k-1}{k}$ be true for any pair of primes $p_{n-1},p_n$ with $p_{n-1}> C$, then

\begin{equation}
p_n - (k - 0.5)p_n^\frac{k-1}{k} < p_{n-1} < p_n
\end{equation}

and $p_{n-1}$ belongs to $(p_n - (k - 0.5)p_n^\frac{k-1}{k}, p_n)$, so this interval contains a prime number.
\end{proof}

\begin{lemma}\label{lem2}
There exists an effectively computable constant $C$ for any fixed real number $k \geq 40/19$ such that for every pair primes $p_{n-1}, p_n$ with $p_{n-1}> C$, the following inequality holds:

\begin{equation}
p_n - p_{n-1} < (k - 0.5) p_n^\frac{k-1}{k}
\end{equation}
\end{lemma}

\begin{proof}
Using Theorem 1 from \cite{Baker:2001aa} (Theorem 1. For all $x > x_0$, the interval $[x -  x^{0.525}, x]$ contains prime numbers. With enough effort, the value of $x_0$ could be determined effectively.) we can claim that any pair of neighbouring primes $p_{n-1}, p_n$ with $p_{n-1} > x_0$, which in turn implies that:
\begin{equation}
p_n - p_{n-1} < p_n^{0.525}
\end{equation}

Thus the inequality (4) is true for any real $k\geq 40/19$. So $C = p_r$ where $p_{r-1}< x_0< p_r$.
\end{proof}

\begin{proof}[Proof of proposition \ref{prop1}]
Proposition \ref{prop1} is true due to lemmas \ref{lem1}, \ref{lem2}
\end{proof}

\begin{proposition}\label{prop2}
An interval $(n - (k - 0.5)n^{\frac{k-1}{k}}, n)$, where $k \geq 40/19$ is any fixed real number, contains a prime $p$ for every integer $n > C$, where $C$ is an effectively computable constant.
\end{proposition}

\begin{lemma}\label{lem3}
Proposition \ref{prop2} is true for every integer $n > C$ if and only if proposition \ref{prop1} is true for every prime $p > C$.
\end{lemma}

\begin{proof}
Let proposition \ref{prop2} be true, then proposition \ref{prop1} is true for every $p$ prime greater than $C$.

Let proposition \ref{prop1} be true for all primes $p \geq p_r$ where $p_{r-1} < C < p_r$, but assume proposition \ref{prop2} is false for some integers. Let $n_0$ be the minimal integer for which proposition \ref{prop2} does not hold. This implies that the interval $(n_0 - (k - 0.5) n_0^\frac{k-1}{k}, n_0)$ contains no prime numbers. Let $p_{n-1}, p_n$ be such primes that $p_r \leq p_{n-1} < n_0 < p_n$. Then the interval $(p_n - (k - 0.5) p_n^\frac{k-1}{k}, p_n)$ contains no primes. Indeed, $(n_0 - (k - 0.5) n_0^\frac{k-1}{k}, p_n) = (n_0 - (k - 0.5) n_0^\frac{k-1}{k}, n_0) \cup [n_0] \cup(n_0, p_n)$ contains no prime numbers. We have $(p_n - (k - 0.5) p_n^\frac{k-1}{k}, p_n) \subset (n_0 - (k - 0.5) n_0^\frac{k-1}{k}, p_n)$ since $n_0 - (k - 0.5) n_0^\frac{k-1}{k} < p_n - (k - 0.5) p_n^\frac{k-1}{k}$. Thus, we can conclude the interval $(p_n- (k - 0.5)p_n^\frac{k-1}{k}, p_n)$ contains no primes, contradicting proposition \ref{prop1}.
\end{proof}

\begin{proof}[Proof of proposition \ref{prop2}]
According to proposition \ref{prop1} there exists such an integer $C$ that for any prime $p_n \geq C, p_n - p_{n-1} < (k - 0.5) p_n^\frac{k-1}{k}$ is true. According to lemma \ref{lem3}, $(n - (k - 0.5)n^\frac{k-1}{k}, n)$ contain primes for all integers $n > C$.
\end{proof}

\begin{proof}[Proof of Theorem \ref{thm1}]
Since proposition \ref{prop2} is true for all integers $n > C$ so $(n - (k - 0.5)n^\frac{k-1}{k}, n)$ contains a prime $p$. Let $n = [m^k]$ with $[m^k] > C$ then an interval $([m^k] - (k - 0.5)[m^k]^{\frac{k-1}{k}}, [m^k])$ contains a prime $q$. Since $([m^k] - (k - 0.5)[m^k]^{\frac{k-1}{k}}, [m^k])\subset((m - 1)^k, m^k)$ therefore $q$ belongs to $((m - 1)^k, m^k)$ and theorem \ref{thm1} is true.
\end{proof}

\subsubsection{Application}

\begin{theorem}\label{thm2}
(A.E. Ingham, 1937). An interval $((n - 1)^3, n^3)$ contains a prime $p$ for every integer $n$ where $n > C$, where $C$ is an effectively computable constant.
\end{theorem}

\begin{proof}
Theorem \ref{thm2} is a particular case of the ``exponential'' theorem where $k = 3 > 40/19$. In this case
\begin{equation}
p_n - p_{n-1} < 2.5 p_n^{0.666}
\end{equation}

Thus according to the ``exponential'' theorem there exists an effectively computable constant $C$ such that for every integer $n > C$ an interval $((n - 1)^3, n^3)$ contains a prime number.
\end{proof}

We would like to note that using an evaluation of the difference between consecutive primes in the form $p_n- p_{n-1}=O(p_{n-1}^\theta)$ where $\theta = 0.75 + \epsilon$ is Tchudakoff's constant we would not have been able to prove the theorem, however with Ingham's constant $\theta = 0.625 + \epsilon$, and the more with the constant $\theta = 0.525$ we already are able to do so, \cite{Pintz:2009aa}.

\begin{theorem}\label{thm3}
An interval $((n - 1)^{\sqrt{5}}, n^{\sqrt{5}})$ contains a prime $p$ for every integer $n > C$, where $C$ is an effectively computable constant.
\end{theorem}

\begin{proof}
Theorem \ref{thm3} is a particular case of the ``exponential'' theorem with $k = \sqrt{5} > 40/19$. In this case the evaluation of the difference between consecutive primes takes the following form
\begin{equation}
p_n - p_{n-1} < 1.736 p_n^{0.552}
\end{equation}
Thus according to the ``exponential'' theorem there exists an effectively computable constant $C$ such that for every integer $n > C$ an interval $((n - 1)^{\sqrt{5}}, n^{\sqrt{5}})$ contains a prime number.
\end{proof}

\subsection{``Fractional'' Theorem}

\begin{theorem}\label{thm4}
An interval $(\frac{k-1}{k}n, \frac{k}{k-1}n)$, where $k \geq 2$ is any fixed real number, contains two prime numbers $p, q$ for every integer $n > C$, where $C$ is an effectively computable constant.
\end{theorem}

\begin{proposition}\label{prop3}
An interval $(n, \frac{k}{k-1}n)$, where $k \geq 2$ is an arbitrary fixed real number, contains a prime number $p$ for every integer $n > A$, where $A$ is an effectively computable constant.
\end{proposition}

\begin{lemma}\label{lem4}
An interval $(p, \frac{k}{k-1}p)$, where $k \geq 2$ is an arbitrary fixed real number, contains a prime number $q$ for every prime $p > A$, where $A$ is an effectively computable constant.
\end{lemma}

\begin{lemma}\label{lem5}
Lemma \ref{lem4} is true if and only if there is such a constant $A$ that for every pair of neighbouring prime numbers $p_{n-1}, p_n$ with $p_{n-1} > A$ the following inequality is satisfied:
\begin{equation}
p_n - p_{n-1} < p_{n-1} /(k-1)
\end{equation}
\end{lemma}

\begin{proof}
Let Lemma \ref{lem4} be true, then the interval $(p_{n-1}, \frac{k}{k-1}p_{n-1})$ contains a prime number $q$ such that $p_{n-1} < p_n \leq q$ and $p_n$ belongs to this interval. Thus we have
\begin{equation}
p_{n-1}< p_n < \frac{k}{k-1}p_{n-1} \text{ and } p_n - p_{n-1} < p_{n-1} /(k-1)
\end{equation}

Let the inequality $p_n - p_{n-1} < p_{n-1} /(k-1)$ be true for any pair $p_{n-1},p_n$ where $p_{n-1} > A$ then:
\begin{equation}
p_{n-1}< p_n< p_{n-1}+ p_{n-1} /(k-1) \text{ and } p_n \text{ belongs to } (p_{n-1}, \frac{k}{k-1}p_{n-1})
\end{equation}

Thus any interval $(p_{n-1}, \frac{k}{k-1}p_{n-1})$ contains a prime number for every prime $p_{n-1} > A$.
\end{proof}

\begin{proposition}\label{prop4}
An interval $(\frac{k-1}{k}n, n)$, where $k \geq 2$ is an arbitrary fixed real number, contains a prime $p$ for every integer $n > B$, where $B$ is an effectively computable constant.
\end{proposition}

\begin{lemma}\label{lem6}
An interval $(\frac{k-1}{k}p, p)$, where $k \geq 2$ is an arbitrary fixed real number, contains a prime $q$ for every prime $p > B$, where $B$ is an effectively computable constant.
\end{lemma}

\begin{lemma}\label{lem7}
Lemma \ref{lem6} is true if and only if there is such a constant $B$ that for every pair of neighbouring prime numbers $p_{n-1}, p_n$ where $p_n > B$ the following inequality is satisfied:
\begin{equation}
p_n - p_{n-1} < p_n / k
\end{equation}
\end{lemma}

\begin{proof}
Let Lemma \ref{lem6} be true, then the interval $(\frac{k-1}{k}p_n, p_n)$ contains a prime $q$ such that $q \leq p_{n-1} < p_n$ and $p_{n-1}$ belongs to this interval. Thus we have
\begin{equation}
(k - 1/k)p_n < p_{n-1}< p_n \text{ and } p_n - p_{n-1} < p_n / k
\end{equation}
Let the inequality $p_n - p_{n-1} < p_n / k$ be true for any pair primes $p_{n-1}, p_n$ with $p_n > B$ then:
\begin{equation} 
p_n - p_n  / k < p_{n-1} < p_n \text{ and } p_{n-1} \text{ belongs to } ((k - 1/k)p_n, p_n)
\end{equation}
Thus any interval $(\frac{k-1}{k}p_n, p_n)$ contains a prime number for every prime $p_n > B$.
\end{proof}

\begin{proposition}\label{prop5}
There exists such an effectively computable constant $C$ that lemma \ref{lem4} and lemma \ref{lem6} are true simultaneously for any pair of consecutive primes $p_{n-1}, p_n$ where $p_{n-1} > C$.
\end{proposition}

\begin{proof}
Let $C$ be such that for any pair primes $p_{n-1}, p_n$ with $p_{n-1} > C$, $p_n - p_{n-1} < p_{n-1}/ k$ is true. Then lemma \ref{lem5} and lemma \ref{lem7} will be satisfied simultaneously and lemma \ref{lem4} and lemma \ref{lem6} are true. Let us show that such an effectively computable constant $C$ exists. Indeed, using Proposition 1.10 from \cite{Dusart:2010aa} (Proposition 1.10. For $n > 463$, $p_{n+1} \leq p_n(1 + 0.5/ln^2 p_n))$ we have $p_n - p_{n-1} \leq p_{n-1}/2ln^2 p_{n-1}$. Thus our problem is to find such $n_0$ that for every integer $n > n_0$ the following inequality holds:
\begin{equation}
p_n - p_{n-1} \leq \frac{p_{n-1}}{2ln^2 p_{n-1}} < \frac{p_{n-1}}{k}.
\end{equation}

Since $ln(p_n)$ is a strictly increasing function so there exists such $n_0$ that for every integer $n > n_0$ this inequality is satisfied. Thus we have the following estimate for $C: C = max(p_r, p_{463}),$ where $p_{r-1} < exp(k^{0.5}) < p_r$.
\end{proof}

\begin{proof}[Proof of proposition \ref{prop3}]
Let proposition \ref{prop3} be true then lemma \ref{lem4} is true for every prime $p > C$. Let lemma \ref{lem4} be true for every prime $p \geq p_r$ where $p_{r-1} < C < p_r$ but proposition \ref{prop3} is false for some integers. Let $n_0 > p_r$ be a minimal integer such that an interval $(n_0, \frac{k}{k-1}n_0)$ contains no primes. Let $p_{n-1}, p_n$ be a pair primes such that $p_r \leq p_{n-1} < n_0 < p_n$ then the interval $(p_{n-1}, \frac{k}{k-1}p_{n-1})$ contains no primes. Indeed, the interval $(p_{n-1}, \frac{k}{k-1}n_0) = (p_{n-1}, n_0)\cup[n_0]\cup(n_0, \frac{k}{k-1}n_0)$ contains no primes. Since $(p_{n-1}, \frac{k}{k-1}p_{n-1}) \subset (p_{n-1}, \frac{k}{k-1}n_0)$ so one also contains no primes which is a contradiction.
\end{proof}

\begin{proof}[Proof of proposition \ref{prop4}]
Proving of proposition 3.13 is analogue of proving of proposition 3.10.
\end{proof}

\begin{proof}[Proof of theorem \ref{thm4}]
According proposition \ref{prop5} there is such an effectively computable constant $C$ that proposition \ref{prop3} and proposition \ref{prop4} are true simultaneously for every integer $n > C$. This means that theorem \ref{thm4} is true for every integer $n > C$. Thus the ``fractional'' theorem is true.
\end{proof}

\subsubsection{Application}

\begin{theorem}\label{thm5}
An interval $((n - 1)^3, n^3)$ contains no less than two primes $p, q$ for every integer $n > C$, where $C$ is an effectively computable constant.
\end{theorem}

\begin{lemma}\label{lem8}
An interval $(\frac{k-1}{k}n, \frac{k}{k-1}n)$, where $k = g^{1.5} / (g^{1.5} - (g - 1)^{1.5})$, $g \geq 3$ is an integer contains at least two primes for every $n > C(g)$ where $C(g)$ is an effectively computable constant.
\end{lemma}

\begin{proof}
Lemma \ref{lem8} is true due to the ``fractional'' theorem. 
\end{proof}

According to \cite{Baker:2001aa} there exists an effectively computable constant $x_0$ such that for every pair consecutive primes $p_{n-1}, p_n$ where $p_{n-1}> x_0$ the following inequality holds:
\begin{equation}
p_n - p_{n-1} < p_n^{0.525}
\end{equation}

\begin{lemma}\label{lem9}
Let $g > x_0$ then there exists such an effectively computable constant $C(g)$ that for every pair of consecutive primes $p_{n-1}, p_n$ where $p_{n-1} > C(g)$, we have the following inequalities
\begin{equation}
p_n - p_{n-1}< p_n^{0.525}< p_{n-1}/(g^{1.5} / (g^{1.5} - (g - 1)^{1.5}))
\end{equation}
\end{lemma}

\begin{proof}
Our goal is to obtain an estimate of $C(g)$ comparing two evaluations of the difference between consecutive primes. Our problem is to define $C(g)$ as a corollary of the inequality
\begin{equation}
p_n^{0.525} < p_{n-1}/(g^{1.5} / (g^{1.5} - (g - 1)^{1.5})
\end{equation}
Thus $(p_n / p_{n-1})^{0.525} (g^{1.5} / (g^{1.5}- (g - 1)^{1.5}) < p_{n-1}^{0.475}$ continuing our calculation we will obtain the constant $C(g)$ which has the form $C(g) = g^2([g^{2/19}] +1)$.Thus the inequality $p_n - p_{n-1} < p_{n-1}/ k$, where $k = g^{1.5} / (g^{1.5} - (g - 1)^{1.5})$ is true for every pair primes $p_{n-1}, p_n$ with $p_{n-1} > C(g)$.
\end{proof}

\begin{lemma}\label{lem10}
The interval $(\frac{k-1}{k}n, \frac{k}{k-1}n)$, where $k = g^{1.5} / (g^{1.5} - (g - 1)^{1.5})$ contains at least two primes for every integer $n > C(g)$, where $C(g) = g^2([g^{2/19}] +1)$ with $g > x_0$.
\end{lemma}

\begin{proof}
Lemma \ref{lem10} is true due to lemma \ref{lem9} and the ``fractional'' theorem.
\end{proof}

\begin{lemma}\label{lem11}
The interval $(C(g), (g(g - 1))^{1.5})$ contains no less five prime numbers.
\end{lemma}

\begin{proof}
According to Ramanujan's evaluations \cite{Ramanujan:1919aa} the interval $(C(g), 2C(g))$ where $C(g) > 20$ contains at least five primes. Let us show that an inequality $2C(g) < (g(g - 1))^{1.5}$ is satisfied. Indeed, the inequality $2C(g) <2(g^{40/19} + g^2) < (g(g - 1))^{1.5}$ takes place already under $g > 20$.
\end{proof}

\begin{proof}[Proof of Theorem \ref{thm5}]
Let us take such an integer $n_0 = (g(g - 1))^{1.5} + \theta$ that $|\theta| \leq \frac{1}{2}$. Since the inequality $(g(g - 1))^{1.5} > 2C(g)$ is true so the integer $n_0 > 2C(g) - 1 > C(g)$. The interval $(C(g), n_0)$ contains not less five prime numbers. Further according to the ``fractional'' theorem an interval $((g - 1/g)^{1.5}n_0, (g/g - 1)^{1.5}n_0)$ contains at least two primes $p, q$. Thus we have the inequality:
\begin{equation}
(g - 1/g)^{1.5} n_0 = (g - 1)^3 + \theta(g - 1/g)^{1.5} < p, q < g^3 + \theta(g /g - 1)^{1.5} = (g /g - 1)^{1.5}n_0
\end{equation}

Since $\max(|\theta|(g - 1/g)^{1.5}, |\theta|(g /g - 1)^{1.5}) < 1$ is true therefore $(g - 1)^3 < p, q < g^3$. Thus an interval $((g - 1)^3, g^3)$ contains at least two prime numbers for every integer $g > x_0$.
\end{proof}

\section{Conclusion}

In this work, we  presented a method proving of Diophantine inequalities with primes through the use of auxiliary inequalities and available evaluations of the difference between consecutive primes. By applying this method we have proved a number of new results: the ``exponential'' and ``fractional'' theorems, Ingham's theorem with two primes; and a fresh proof of  proof of Ingham's theorem with effectively computable constant. The method of proving Legendre's and Oppermann's conjectures as well as Bertrand's postulate using auxiliary inequalities and expected evaluations of the difference between consecutive primes \cite{Pintz:2009aa} is described in \cite{Sidokhine:2014aa}, \cite{Sidokhine:2015aa} and \cite{Sidokhine:2015ab}.

\bibliography{references}

\begin{thebibliography}{1}

\bibitem{Baker:2001aa}
R.C. Baker, G.~Harman, and J.~Pintz.
\newblock The difference between consecutive primes ii.
\newblock {\em Proc. London Math. Soc.}, 83:532--562, 2001.

\bibitem{Dusart:2010aa}
Pierre Dusart.
\newblock Estimates of some functions over primes without r.h.
\newblock arXiv:1002.0442, 2010.

\bibitem{Pintz:2009aa}
J{\'a}nos Pintz.
\newblock Landau's problems on primes.
\newblock {\em Journal de Th{\'e}orie des Nombres de Bordeaux},
  21(2):357--4014, 2009.

\bibitem{Ramanujan:1919aa}
S.~Ramanujan.
\newblock A proof of bertrand's prostulate.
\newblock {\em J. Indian Math. Soc.}, (11):181--182, 1919.

\bibitem{Sidokhine:2014aa}
Felix Sidokhine.
\newblock Diophantine inequalities with primes as a problem of the difference
  between consecutive primes.
\newblock arXiv: 1410.6856v1, 2014.

\bibitem{Sidokhine:2015aa}
Felix Sidokhine.
\newblock On proving of diophantine inequalities with prime numbers by
  evaluations of the difference between consecutive primes.
\newblock arXiv: 1507.07025v1, 2015.

\bibitem{Sidokhine:2015ab}
Felix Sidokhine.
\newblock On the difference between consecutive primes and estimates of the
  number of primes in the interval $(n, 2n)$.
\newblock arXiv: 1507.07028v1, 2015.

\end{thebibliography}
\bibliographystyle{plain}

\end{document}